\documentclass[a4paper,10pt]{amsart}

\usepackage{amsthm}
\usepackage{amssymb}
\usepackage{amsmath}
\usepackage{mathtools}
\usepackage{pdfpages}
\usepackage{changes}
\usepackage{exscale,cite,color,amsopn}
\usepackage[colorlinks=true,pdftex,unicode=true,linktocpage,bookmarksopen,hypertexnames=false]{hyperref}
\usepackage{colortbl}
\usepackage{enumitem}

\usetikzlibrary{trees}

\renewcommand{\leq}{\leqslant}
\renewcommand{\geq}{\geqslant}

\DeclareMathOperator{\Sym}{Sym}

\newcommand{\Z}{\mathbb{Z}}

\makeatletter
\numberwithin{equation}{section}
\numberwithin{figure}{section}
\numberwithin{table}{section}
\newtheorem{thm}{Theorem}[section]
\newtheorem*{thm*}{Theorem}
\newtheorem{lem}[thm]{Lemma}

\newtheorem{pro}[thm]{Proposition}

\theoremstyle{definition} 
\newtheorem{defn}[thm]{Definition}
\newtheorem{problem}[thm]{Problem}

\newtheorem{exa}[thm]{Example}

\makeatother

\title{On the enumeration of finite $L$-algebras}
\author{C. Dietzel, P. Mench\'on, L. Vendramin}

\address{Department of Logic, Nicolaus Copernicus University in Toru\'n,\linebreak 
Fosa Staromiejska 1a, 87-100 Toru\'n, Poland}
\email{paula.menchon@v.umk.pl}

\address{IMAS--CONICET and Depto. de Matem\'atica, FCEN, Universidad de Buenos Aires, Pab.~1, Ciudad Universitaria, C1428EGA, Buenos Aires, Argentina}
\email{lvendramin@dm.uba.ar}

\address{Department of Mathematics and Data
Science, Vrije Universiteit Brussel, Pleinlaan 2, 1050 Brussel}
\email{carstendietzel@gmx.de}
\email{Leandro.Vendramin@vub.be}



\begin{document}

\begin{abstract}
    We use Constraint Satisfaction Methods to construct and enumerate 
    finite $L$-algebras up to isomorphism. These objects were recently
    introduced by Rump and appear
    in Garside theory, algebraic logic, and the study of the combinatorial
    Yang--Baxter equation. There are  
    377,322,225 isomorphism classes of $L$-algebras of size eight. 
    The database constructed suggest
    the existence of bijections between certain classes of $L$-algebras
    and well-known combinatorial objects. We prove that Bell numbers enumerate isomorphism classes of finite linear $L$-algebras. We also prove that finite regular $L$-algebras are
    in bijective correspondence with infinite-dimensional Young diagrams. 
\end{abstract}

\maketitle

\section{Introduction}
 
In the study of set-theoretical solutions of the Yang--Baxter equation, 
the \emph{cycloid equation},  
\[
(x\cdot y)\cdot (x\cdot z)=
(y\cdot x)\cdot (y\cdot z), 
\]
plays a fundamental role, see for example \cite{MR3374524,MR2442072,MR2132760}. 
The formula first occurred in algebraic logic \cite{MR292965,MR403432,MR944033} 
and also appears 
in the theory of Garside groups \cite{MR3362691} and in connection with
von Neumann algebras and orthomodular lattices \cite{MR4373225,MR3824801}. To combine 
similarities between these different objects, Rump 
introduced $L$-algebras in \cite{MR2437503}. 
It turns out that
$L$-algebras appear in connection with other topics in 
algebraic logic \cite{MR3012376}, 
in the combinatorial study of the Yang--Baxter equation \cite{MR2437503} 
and in number theory \cite{MR3805465}. 

An $L$-algebra is a set $X$ together with a binary operation 
$(x,y)\mapsto x\cdot y$ 
on $X$ 
that satisfies 
the cycloid equation and such that 
there exists a distinguished element $e\in X$ with some particular 
properties (see Definition \ref{defn:L} for details). The distinguished
element $e$ is crucial
to define a partial order on $X$:
\begin{equation}
    \label{eq:poset_intro}
    x \leq y \Longleftrightarrow x \cdot y = e.
\end{equation}

If $P$ is a partially ordered set (\emph{poset}, for short), we call an $L$-algebra $X$ 
with the same underlying set as $P$ an \emph{$L$-algebra on $P$}, when
\eqref{eq:poset_intro} 
holds for all $x,y \in P$. 
In this case, 
$P$ has a greatest element that coincides with the logical unit of $X$.

$L$-algebras have everything
to become fundamental objects in algebraic logic. 
On the one hand, when the binary operation of $X$ is interpreted as 
implication and
equality as logical equivalence between propositions, $L$-algebras
provide a significant extension of classical logic. For instance, the cycloid equation holds in classical and intuitionistic logic, and also in logics based on a continuous t-norm. On the other hand, 
$L$-algebras have a rich algebraic structure that provides
new tools to study logical structures. For example, every 
$L$-algebra 
$X$ admits a structure group $G(X)$, and this group can be used to 
give a conceptual 
and short proof of Mundici's embedding theorem \cite{MR819173}.

One of our main results here is an explicit construction of 
isomorphism classes of $L$-algebras 
of small size.
For small size, the approach is somewhat 
inspired by \cite{MR4405502} 
and the cycloid equation. 

\begin{thm}
\label{thm:upto7}
Up to isomorphism, there are 924,071 $L$-algebras of size $\leq7$.
\end{thm}

Since Hilbert algebras \cite{MR0199086} are particular
cases of $L$-algebras, we also obtain the construction of
small Hilbert algebras up to isomorphism. 

\begin{thm}
\label{thm:Hilbert}
    Up to isomorphism, there 
    are 1,085,642 
    Hilbert algebras of size $\leq10$. 
\end{thm}

The method developed
to prove Theorem \ref{thm:upto7} cannot be applied for $L$-algebras of size eight. This happens because the search space concerning $L$-algebras of size eight is 
``too large". Additionally, checking lexicographical symmetry breaking for the whole symmetric group $\Sym_{n-1}$ is already computationally expensive for $n=8$. For these reasons, the enumeration
problem of $L$-algebras of size eight requires new ideas. 

\begin{thm}\
\label{thm:eight}
    Up to isomorphism, there are 377,322,225 
    $L$-algebras of size eight. 
\end{thm}

A detailed account regarding the number of 
isomorphism classes of $L$-algebras and Hilbert algebras of size $n$ 
appears in Tables \ref{tab:L_algebras}
and \ref{tab:hilbert}.

\begin{table}[h]
\caption{The number of isomorphism classes of (discrete) 
$L$-algebras of size $\leq8$.}
\begin{tabular}{|r|cccccccc|}
\hline
$n$ & 1 & 2 & 3 & 4 & 5 & 6 & 7 & 8\tabularnewline
\hline 
$L(n)$ & 1 & 1 & 5 & 44 & 632 & 15582 & 907,806 & 377,322,225\tabularnewline
discrete & 1 &1  & 3 & 15 & 101 & 1241 & 151,846 & 149,390,095\tabularnewline
\hline
\end{tabular}
\label{tab:L_algebras}
\end{table}

\begin{table}
\caption{The numbers of isomorphism classes of Hilbert algebras of size $\leq10$.}
\begin{tabular}{|r|cccccccccc|}
\hline
$n$ & 1 & 2 & 3 & 4 & 5 & 6 & 7 & 8 & 9 & 10\tabularnewline
\hline 
Hilbert & 1 & 1 & 2 & 6 & 21 & 95 & 550 & 4036& 37602 & 1,043,328\tabularnewline
\hline
\end{tabular}
\label{tab:hilbert}
\end{table}

As a by-product, we produce a database of finite $L$-algebras and
Hilbert algebras. 

The proofs of Theorems \ref{thm:upto7}, 
\ref{thm:Hilbert} and \ref{thm:eight} use constraint programming techniques. 
Constraint programming
is a paradigm for solving combinatorial problems. The idea is
to search for variables 
that satisfy a certain number of constraints.

The main tools used are the constraint modeling assistant 
Savile Row \cite{nightingale2017automatically} and 
\cite{gap4}. Savile Row is a flexible modelling assistant 
for constraint programming; it provides a high-level language
fitting perfectly with our particular combinatorial problem.

The databases produced are available at 
\url{https://github.com/vendramin/L}, 
with DOI name \href{https://doi.org/10.5281/zenodo.6630229}{10.5281/zenodo/6630229}.

The calculations of Theorems \ref{thm:upto7}, \ref{thm:Hilbert}
and \ref{thm:eight} were
performed 
in an 
Intel(R) Core(TM) i7-8700 CPU @ 3.20GHz, with 32GB RAM. 

\begin{table}[ht]
\caption{The construction of isomorphism classes of 
$L$-algebras of size eight.}
\begin{tabular}{|r|c|cc|}
\hline
\multicolumn{1}{|l|}{} & discrete  & \multicolumn{2}{c|}{non-discrete}         \\ \hline
Poset                  & trivial   & \multicolumn{1}{c|}{diamond}   & other    \\ \hline
Nr. cases              & 73        & \multicolumn{1}{c|}{156}       & 2043     \\ \hline
Nr. solutions          & 149,390,095 & \multicolumn{1}{c|}{138,219,543} & 89,712,587 \\ \hline
Run-time               & 4 days    & \multicolumn{1}{c|}{5 hours}   & 11 hours \\ \hline
Data (compressed)      & 208MB     & \multicolumn{1}{c|}{212MB}     & 153MB    \\ \hline
\end{tabular}
\label{tab:runtime}
\end{table}

By inspection of the database, we find connections involving certain finite families of 
$L$-algebras and well-known combinatorial objects. 
Two families of $L$-algebras are particular relevant 
for this work. 
We prove that finite \emph{linear}  
$L$-algebras (Definition \ref{defn:linear}) are enumerated by Bell numbers 
(Theorem \ref{thm:linear}). 
We also prove that the elements of 
another important family of
finite $L$-algebras, that of \emph{regular} $L$-algebras (Definition \ref{defn:regular}),  
are in bijective correspondence with 
infinite-dimensional Young diagrams (Theorem~\ref{thm:regular}). 


\section{Preliminaries}
\label{preliminaries}


We refer to \cite{MR2437503} for the basic theory of $L$-algebras. 

\begin{defn}
\label{defn:L}
An $L$-algebra is a set $X$ together with a function $X\times X\to X$,
$(x,y)\mapsto x\cdot y$, such that
\begin{enumerate}
    \item there exists an element $e\in X$ such that $e\cdot x=x$ and $x\cdot e=x\cdot x=e$ for all $x\in X$, 
    \item $(x\cdot y)\cdot (x\cdot z)=(y\cdot x)\cdot (y\cdot z)$ holds for all $x,y,z\in X$, and
    \item if $x\cdot y=y\cdot x=e$, then $x=y$. 
\end{enumerate}
\end{defn}

The element $e$ is unique and is called the \emph{logical unit} of $X$. 

\begin{defn}
A \emph{Hilbert algebra} is an $L$-algebra $X$ such that 
\begin{equation}
    \label{eq:Hilbert}
    x\cdot (y\cdot z)=(x\cdot y)\cdot (x\cdot z)
\end{equation}
holds for all $x,y,z\in A$.  
\end{defn}

\begin{defn}
If $X$ and $Y$ are $L$-algebras, 
a homomorphism of 
$L$-algebras is a map $f\colon X\to Y$ such that
$f(x\cdot y)=f(x)\cdot f(y)$ for all $x,y\in X$.
\end{defn}

The derived notions of \emph{isomorphism} and \emph{automorphism} are clear from this definition.
If $f\colon X\to Y$ is a homomorphism of $L$-algebras,
then $f(e)=e$. It follows that the 
automorphism group of an $L$-algebra
$X$ of size $n$ is isomorphic to a subgroup of $\Sym_{n-1}$, the symmetric
group on $n-1$ letters. 

It can be shown \cite[Proposition 2]{MR2437503}
that for an $L$-algebra $X$, the relation $\leq$ given~by
\begin{equation}
    \label{eq:poset}
    x \leq y \Longleftrightarrow x \cdot y = e
\end{equation}
is a partial order relation. We will call this relation the \emph{natural} partial order relation on $X$.
Note that from this definition we immediately see that $e$ is the unique maximal element of $X$.

Given elements $x,y,z \in X$ with $y \leq z$, we can calculate
\[
(x \cdot y) \cdot (x \cdot z) = (y \cdot x) \cdot (y \cdot z) = (y \cdot x) \cdot e = e,
\]
which implies that $x \cdot y \leq x \cdot z$, as well, so the assignments $y \mapsto x \cdot y$ are monotone for all $x \in X$.

One case of natural order is of particular importance in this article.

\begin{defn}
\label{defn:discrete}
    An $L$-algebra $X$ is said to be \emph{discrete} if 
    $x \leq y $ if and only if $x=y$ or $y=e$. 
\end{defn}

This means, an $L$-algebra is discrete when the order on $X \setminus \{ e \}$ is discrete.


%

\subsection{A combinatorial description of finite $L$-algebras}

Let $n\geq2$ and $X$ be a finite $L$-algebra of size $n$. 
Without loss of generality we may assume that $X=\{1,\dots,n\}$, where   
$n$ is the logical unit. The $L$-algebra structure on $X$ is then 
described by the 
matrix $M=(M_{i,j})\in\Z^{n\times n}$, where
\[
    M_{i,j}=i\cdot j,\quad i,j\in\{1,\dots,n\}.
\]
The matrix $M$ satisfies the following conditions:
\begin{enumerate}
	\item $M_{n,j}=j$ for all $j\in\{1,\dots,n\}$, 
	\item $M_{i,n}=n$ for all $i\in\{1,\dots,n\}$, 
	\item $M_{k,k}=n$ for all $k\in\{1,\dots,n\}$, and  
	\item $M_{M_{i,j},M_{i,k}}=M_{M_{j,i},M_{j,k}}$ holds for all $i,j,k\in\{1,\dots,n\}$. 
	\item if $M_{i,j}=n$ and $M_{j,i}=n$, then $i=j$, 
\end{enumerate}
There is a bijective correspondence between 
$L$-algebras and matrices satisfying Conditions (1)--(5); 
see Example \ref{exa:L} for a concrete
example of a matrix representation of an $L$-algebra.

Since $n$ is the logical unit, $g(n)=n$ holds for each automorphism $g$ of the $L$-algebra. 
Over the set of $n\times n$ matrices satisfying Conditions (1)--(5) 
we define the equivalence relation given by
\[
M\sim N\Longleftrightarrow \text{there exists }g\in\Sym_{n-1}\text{ such that }
N_{i,j}=g^{-1}\left(M_{g(i),g(j)}\right)
\]
for all $i,j\in\{1,\dots,n\}$.
It follows that 
the $L$-algebras $X$ and $Y$ are isomorphic if and only if
$M_X\sim M_Y$. 

\begin{exa}
\label{exa:L}
    Let $X$ be the $L$-algebra structure over $\{x,y,e\}$ 
    with logical unit $e$ given by
    \[
    e\cdot y=y,\quad 
    x\cdot y=y\cdot x=e\cdot x=x.
    \]
    Let $f\colon \{1,2,3\}\to\{x,y,e\}$ be the bijective
    map given by $f(1)=x$, $f(2)=y$ and $f(3)=e$.  
    Then $X$ can be represented as the matrix 
    \[
    M_X=\begin{pmatrix}
    3 & 1 & 3\\
    1 & 3 & 3\\
    1 & 2 & 3
    \end{pmatrix}.
    \]
    Note that, in this new representation of $X$, 
    the logical unit is the element $3$. 
\end{exa}

\section{Proofs of Theorems \ref{thm:upto7}, \ref{thm:Hilbert} and \ref{thm:eight}}

\subsection{Proof of Theorems \ref{thm:upto7} and \ref{thm:Hilbert}}

To construct all $L$-algebras of size $n\leq 7$ we need to find all possible 
matrices $M\in\Z^{n\times n}$ with
coefficients in $\{1,2,\dots,n\}$ that satisfy Conditions (1)--(5). 
We use 
Savile Row 
\cite{nightingale2017automatically} to find all matrices with entries in $\{1,\dots,n\}$ 
satisfying Conditions (1)--(5). 

To remove repetitions we proceed as follows.  
If $g\in\Sym_{n-1}$ and $M$ is a matrix, we
denote by $M^g$ the matrix given by
\[
	(M^g)_{i,j}=g^{-1}\left(M_{g(i),g(j)}\right)
\]
for all $i,j\in\{1,\dots,n\}$.
We need to find those matrices
$M$ such that
\begin{equation}
	\label{eq:lex}
	M\leq_{\operatorname{lex}} M^g
\end{equation}
for all $g\in\Sym_{n-1}$, where
$A\leq_{\operatorname{lex}}B$ if and only if  
\begin{multline*}
(A_{1,1},A_{1,2},\dots,A_{1,n},A_{2,1},A_{2,2},\dots,A_{n,n})\\
\leq
(B_{1,1},B_{1,2},\dots,B_{1,n},B_{2,1},B_{2,2},\dots,B_{n,n})
\end{multline*}
with lexicographical order.  

Our method produces a database (159MB) of isomorphism classes of
$L$-algebras of size $\leq7$ in about 6 hours. 

Minor variations of our algorithm produce database of particular classes
of $L$-algebras such as Hilbert algebras. A database (11MB) 
of isomorphism classes of Hilbert algebras
of size $\leq9$ was constructed in about two hours. 
Isomorphism classes of Hilbert algebras of size
10 were only enumerated, not constructed. 

\subsection{Proof of Theorem \ref{thm:eight}}

The approach used in the proof of Theorem \ref{thm:upto7} does not work
for $L$-algebras of size eight. As is already indicated in the introduction, the main problems we are confronted with are: 
\begin{enumerate}
\item The vast size of the search space for $n=8$: from the $64$ entries $M_{i,j}$, only $22$ (the last row, the last column and the diagonal) follow from the $L$-algebra axioms. Therefore, the search space contains $8^{42}$ elements.
\item The lexicographical symmetry breaking is very expensive as there are potentially $7! = | \Sym_{8-1} |$ permutations to consider in order to determine if a matrix is lexicographically least in its conjugacy class. However, there is a significant amount of $L$-algebras that have a trivial or, at least, small automorphism group. An elimination of the lexicographical symmetry breaking for some of these asymmetric $L$-algebras is therefore expected to accelerate the search process significantly.
\end{enumerate}

Since 
each $L$-algebra has the structure of a poset under \eqref{eq:poset}, 
we can 
split the problem of constructing all non-isomorphic $L$-algebras of size $n$ into the construction 
of all non-isomorphic $L$-algebras on a given poset $(X,\leq)$. 

Note that a poset of size $n$ on the set $\{1,\ldots,n \}$ with maximal element $n$ can be described by a matrix $P = (P_{i,j}) \in \{0,1 \}^{(n-1) \times (n-1)}$, where
\[
P_{i,j} = 1 \Longleftrightarrow i \leq j.
\]
Given a matrix $P$ describing the poset $(X,\leq)$, 
the condition that a matrix $M_{i,j}$ is an $L$-algebra on $(X,\leq)$ can then be rewritten as
\[
M_{i,j} = n \Longleftrightarrow P_{i,j}=1 \textnormal{ for all } i,j \in \{1,\ldots,n-1 \},
\]
which can easily be translated to a constraint.

To split the search space of all $L$-algebras of size $n$, we first compute 
all non-isomorphic posets on a set of size $n-1$ and 
then compute the non-isomorphic $L$-algebras on each of them separately.
Here, we slightly modify the lexicographic condition by not searching for the lexicographically smallest matrix with respect to all symmetries of the set $\{1, \ldots,n-1 \}$ but only with respect to automorphisms of the poset $(X,\leq)$ which, in cases of posets with few automorphisms, accelerates the search significantly. Note that this does not necessarily lead to a multiplication table that is the lexicographically least with respect to all of $\Sym_{n-1}$.

Unfortunately, searching by posets is not enough when computing $L$-algebras of size eight, as in the case of a discrete $L$-algebra we do not have enough restrictions for the entries in a multiplication table and also there are too many symmetries to consider.

In order to explain our approach, we need to repeat the definition of 
$\wedge$-semibraces \cite[Definition 3]{MR2442072}
and introduce the rather technical notion of a \emph{separating $k$-approximate semilattice}.

\begin{defn}
A $\wedge$-semibrace is a set $X$ with a distinguished element $e \in X$ and binary operations $\wedge, \cdot: X \times X \to X$, such that the pair $(X,\wedge)$ is a commutative monoid with neutral element $e$ and the following identities hold 
for all $x,y,z\in X$: 
\begin{enumerate}
    \item $e \cdot x = x$,  
    \item $x \cdot e = e$, 
    \item $x \cdot (y \wedge z) = (x \cdot y) \wedge (x \cdot z)$, and 
    \item $(x \wedge y) \cdot z = (x \cdot y) \cdot (x \cdot z)$.
\end{enumerate}
\end{defn}

We remark here that the commutativity of the monoid $(X,\wedge)$, together with Equation (4), implies that the pair $(X, \cdot)$ fulfils the cycloid equation.

\begin{thm}[Rump]
\label{thm:Rump2}
Each $L$-algebra $X$ has a unique $\wedge$-closure $C(X)$ such that 
\begin{enumerate}
    \item $C(X)$ is a semibrace that is also an $L$-algebra with logical unit $e$,
    \item $C(X)$ contains $X$ as a sub-$L$-algebra, and 
    \item the monoid $(C(X), \wedge,e)$ is a semilattice generated by the subset $X$.
\end{enumerate}
\end{thm}

\begin{proof}
    See \cite[Theorem 3]{MR2442072}. 
\end{proof}

We note here that, as $e$ is the top element of $X$, the subset $X^{\prime} = X \setminus \{ e \}$ generates the subsemilattice $C(X) \setminus \{ e \}$. 
Using Theorem \ref{thm:Rump2}, the most obvious approach would now be to first construct all $\wedge$-semilattices generated by an $(n-1)$-element set $X^{\prime}$ and then search for all $L$-algebras on $X := X^{\prime} \cup \{ e \}$ with logical unit $e$ that are consistent with the $\wedge$-structure in a sense that we will now explain:

Every semilattice is also a poset under the relation 
\[
x \leq y \Longleftrightarrow x \wedge y = x.
\]
As $C(X)$ is generated by $X$, each element $y \in C(X)$ can be written as
\[
y = \bigwedge \{ x \in X^{\prime}: y \leq x \},
\]
where $\bigwedge \emptyset = e$. Therefore, each element $y \in C(X)$ can be characterized by the set of all elements above $y$ that are lying in $X^{\prime}$.

Note that expressions of the form $(x_1 \wedge x_2 \wedge \ldots \wedge x_k) \cdot z$ with $x_i \in X$ ($i=1, \ldots,k$) and $z \in X$ can be rewritten in terms of the $L$-algebra operation only. For example, 
\begin{align*}
    (x_1 \wedge x_2) \cdot z & = (x_1 \cdot x_2) \cdot (x_1 \cdot z), \\
    (x_1 \wedge x_2 \wedge x_3) \cdot z & = ((x_1 \wedge x_2) \cdot x_3) \cdot ((x_1 \wedge x_2) \cdot z) \\
    & = ((x_1 \cdot x_2) \cdot (x_1 \cdot x_3)) \cdot ((x_1 \cdot x_2) \cdot (x_1 \cdot z)).
\end{align*}
Given a semilattice $(C,\wedge)$ generated by a set $X^{\prime}$, we call an $L$-algebra operation $\cdot$ on $X = X^{\prime} \cup \{ e \}$ \emph{consistent} with the semilattice if the rewritten form of all terms of the form $(x_1 \wedge x_2 \wedge \ldots \wedge x_k) \cdot z$ in terms of $\cdot$ is equal to $e$ if and only if $x_1 \wedge x_2 \wedge \ldots \wedge x_k \leq z$ holds in $C$.

We have seen that if $C$ is a bounded semilattice with top element $e$ where the subsemilattice $C \setminus \{ e \}$ is generated by the subset $X^{\prime}$, then each element $\bigwedge A \in C$, where $A \subseteq C$, can be identified from the set
\[
\overline{A} = \left\{ x \in X^{\prime} : \left( \bigwedge A \right) \leq x \ \right\}.
\]
Note that $\overline{(\bigwedge A) \wedge (\bigwedge B)} = \overline{ \overline{A} \cup \overline{B} }$, so all information about the semilattice $C$ is essentially contained in the operation $A \mapsto \overline{A}$.

This map has the following properties:
\begin{enumerate}
    \item $\overline{\emptyset} = \emptyset$.
    \item for $A \subseteq X^{\prime}$, we have $A \subseteq \overline{A}$.
    \item for $A,B \subseteq X^{\prime}$, we have the implication $A \subseteq \overline{B} \Rightarrow \overline{A} \subseteq \overline{B}$.
    \item The mapping $A \mapsto \overline{A}$ is \emph{separating points}, meaning that for $x,y \in X^{\prime}$, we have the implication $x \neq y \Rightarrow \overline{\{ x\}} \neq \overline{\{ y\}}$.

\end{enumerate}
A map $\mathrm{cl}: \mathcal{P}(X^{\prime}) \to \mathcal{P}(X^{\prime}), A \mapsto \overline{A}$ with these four properties above will be called a \emph{separating closure operator} on $X^{\prime}$. So a possible approach to the enumeration resp. calculation of all $L$-algebras of size $n$ would be to first construct all semilattices on $n-1$ generators by determining the different separating closure operators on a set $X^{\prime}$ of size $n-1$, up to a relabelling of the generators, and then find, for each semilattice, the non-isomorphic $L$-algebra operations on $X = X^{\prime} \cup \{ e \}$ which are consistent with the semilattice structure. Note that this approach in fact splits the search case as the $\wedge$-closure is an invariant of an $L$-algebra.
However, there are two obstacles when trying to use this approach directly:
\begin{enumerate}
    \item There are way too many semilattices generated by an $n$-element set, even for small values of $n$.
    \item The length of the expressions $(x_1 \wedge x_2 \wedge \ldots \wedge x_k) \cdot z$ contains $2^k$ terms $x_i$ or $z$, each one surrounded by $k-1$ brackets. In particular, these expressions do not have fixed length.
\end{enumerate}

In order to deal with this problem, we decided to only compute part of the semilattice structures and demand consistency with only this part of the semilattice structure. This leads to the notion of a \emph{separating k-approximate semilattice}. To define this, we clarify some notation: by $\binom{X}{\leq k}$, we denote the set of all subsets of a set $X$ whose size is $k$ or less.

\begin{defn}
Let $k \geq 1$ and $X$ be a set. A $\emph{separating $k$-approximate semilattice}$ on $X$ is a map $\mathrm{cl}: \binom{X}{\leq k} \to \mathcal{P}(X), A \mapsto \overline{A}$ with the following properties:
\begin{enumerate}
    \item $\overline{\emptyset} = \emptyset$.
    \item For $A \in \binom{X}{\leq k}$, we have $A \subseteq \overline{A}$.
    \item For $A,B \in \binom{X}{\leq k}$, we have the implication $A \subseteq \overline{B} \Rightarrow \overline{A} \subseteq \overline{B}$.
    \item For $x,y \in X$, we have the implication $x \neq y \Rightarrow \overline{\{ x\}} \neq \overline{\{ y\}}$.
\end{enumerate}
\end{defn}

\begin{defn}
If $\mathrm{cl}_1, \mathrm{cl}_2$ are two separating $k$-approximate semilattices on a set $X$, then they are said to be \emph{conjugate}, if there exists a bijection $\pi: X \to X$, such that $\mathrm{cl}_2(\pi(A)) = \pi(\mathrm{cl}_1(A))$ for all $A \in \binom{X}{\leq k}$.
\end{defn}

Note that each separating $1$-approximate semilattice $\mathrm{cl}$ on $X$ determines a partial order on $X$ by declaring $x \leq y : \Longleftrightarrow y \in \mathrm{cl}(\{ x\})$. Also, each partial order on $X$ gives rise to a separating $1$-approximate semilattice by defining $\mathrm{cl}(\emptyset) = \emptyset$ and $\mathrm{cl}(\{ x\}) = \{ y \in X: y \geq x \}$. If $k = |X|$, then a $k$-approximate semilattice is just a closure operator on $X$. So a separating $k$-approximate semilattice on a set $X$ can be seen as something in between a poset on $X$ and a semilattice generated by $X$.

Given a $k$-approximate semilattice on a set $X^{\prime}$ we call an $L$-algebra operation on the set $X = X^{\prime} \cup \{ e \}$ - with $e$ being the logical unit - \emph{consistent} with the approximate semilattice if for $l \leq k$ and elements $x_1, \ldots, x_l,y \in X^{\prime}$, the expression $(x_1 \wedge \ldots \wedge x_l) \cdot y$ evaluates to $e$ (when rewritten in terms of the $L$-algebra operation) if and only if $y \in \overline{\{ x_1, \ldots, x_l \}}$.

For size 8, we use the lexicographical order of matrices with respect to the ordering $(A_{1,1},A_{1,2},A_{2,2},A_{2,1},A_{1,3},A_{2,3},A_{3,3},A_{3,2},A_{3,1}, \ldots, A_{n,n})$, see Figure \ref{fig:matrix_ordering}. For bigger multiplication tables, this ordering is better suited for checking lexicographical symmetry breaking, as most of its initial segments are contained in smaller squares that are invariant under permutations in $\Sym_k$ ($1 \leq k \leq n-1$).

\begin{figure}
\begin{tikzpicture}
\tikzstyle{keepstyle} =[rectangle, rounded corners, draw,  fill=white]
\node (31) at (0,3) {$\bullet$};
\node (21) at (0,4) {$\bullet$};
\node (11) at (0,5) {$\bullet$};
\node (12) at (1,5) {$\bullet$};
\node (22) at (1,4) {$\bullet$};
\node (32) at (1,3) {$\bullet$};
\node (13) at (2,5) {$\bullet$};
\node (23) at (2,4) {$\bullet$};
\node (33) at (2,3) {$\bullet$};
\node (12) at (1,5) {$\bullet$};
\node (13) at (2,5) {$\bullet$};
\draw [-latex, thick] (12) -- (22);
\draw [-latex, thick] (22) -- (21);
\draw [-latex, thick] (13) -- (23);
\draw [-latex, thick] (23) -- (33);
\draw [-latex, thick] (33) -- (32);
\draw [-latex, thick] (32) -- (31);
\node at (2.5,5) {$\cdots$};
\node at (2.5,4) {$\cdots$};
\node at (2.5,3) {$\cdots$};
\node at (0,2.5) {$\vdots$} ;
\node at (1,2.5) {$\vdots$} ;
\node at (2,2.5) {$\vdots$} ;
\node at (2.5,2.5) {$\ddots$} ;

\end{tikzpicture}
\caption{Another ordering of matrix entries.}
\label{fig:matrix_ordering}
\end{figure}
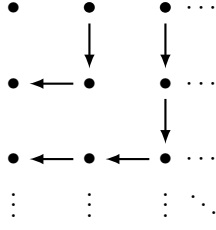 

We now have all tools at our hands to explain how we get our results.

\subsection{Discrete $L$-algebras}

\begin{pro}
\label{pro:discrete}
Up to isomorphism, there are 
149,390,095 discrete 
$L$-algebras of size eight. 
\end{pro}

To deal with the discrete case in size eight, 
we make use of the following result. For $x,y\in X$, write $x\prec y$ to denote that $x< y$ and there
is no $z\in X$ such that $x<z<y$.  

\begin{thm}[Rump]
If $X$ is a discrete $L$-Algebra then its $\wedge$-closure $C(X)$ has the property that $X \setminus \{ e \} = \{ y \in C(X) : y \prec e \}$ and is a lower semimodular lattice meaning that $C(X)$ a lattice under the semilattice order the and the implication $a \prec a \vee b \Rightarrow a \wedge b \prec b$ holds in $C(X)$.
\end{thm}

\begin{proof}
    This is implicitly contained in the proof of \cite[Proposition 18]{MR3373381}. 
\end{proof}

In particular, that means that for $e \neq x \in X$ and $y \in C(X)$ 
with $y \not\leq x$, we have that $x \wedge y \prec y$.
It turns out that splitting the search space by searching for $L$-algebras 
consistent with separating $3$-approximate semilattices fulfilling a 
restricted version of lower semimodularity is the best way of constructing all non-isomorphic discrete $L$-algebras of size eight. 
Therefore, our approach is as follows:

\begin{enumerate}
    \item Construct all non-conjugate $3$-approximate semilattices on  $\{1,\ldots,n-1\}$ that fulfill a restricted version of lower semimodularity.
    \item For each of these $3$-approximate semilattices on $\{1,\ldots,n-1\}$, search for all non-isomorphic $L$-algebras on $\{1,\ldots,n\}$ that are consistent with the approximate semilattice structure.
\end{enumerate}

For each step, we now describe the constraints:

\subsubsection{Constructing separating $3$-approximate semilattices}

We model a separating $3$-approximate semilattice 
$\mathrm{cl}\colon\binom{X^{\prime}}{3} \to \mathcal{P}(X^{\prime})$ as a map 
$\colon\{1, \ldots,n-1 \}^4 \to \{0,1\}$, where
\[
S(x_1,x_2,x_3,y) = 1 \Longleftrightarrow y \in \mathrm{cl}(\{ x_1,x_2,x_3 \}).
\]
This leads to the following set of properties:
\begin{enumerate}
    \item $S(x_1,x_2,x_3,y) = S(x_{\pi(1)},x_{\pi(2)},x_{\pi(3)},y)$ for all $\pi \in \Sym_3$.
    \item $S(x_1,x_2,x_3,x_i)=1$ for $i \in \{1,2,3 \}$.
    \item If $S(x_1,x_2,x_3,y_i)=1$ holds for all $i \in \{ 1,2,3 \}$ and $S(y_1,y_2,y_3,z) =1$, then $S(x_1,x_2,x_3,z) = 1$. 
\end{enumerate}

By Rump's theorem on the lower semimodularity of $C(X)$, we can add the following property which makes it possible to split the search in a small number of cases:

\begin{enumerate}
\item[(4)] If $S(x_1,x_2,x_3,y)=1$, $S(x_1,x_2,x_2,y)=0$ and $S(x_1,x_2,x_3,z)=1$ hold, then $S(x_1,x_2,y,z) = 1$.
\end{enumerate}

This means
\[
y \in \mathrm{cl}(\{x_1,x_2,x_3\})\setminus  \mathrm{cl}(\{x_1,x_2\})
\implies
\mathrm{cl}(\{x_1,x_2,x_3\}) = \mathrm{cl}(\{x_1,x_2,z\}).
\]

Also, for symmetry breaking, we demand that $S$ has to be lexicographically \emph{greatest} with respect to relabelling of the variables to ensure that we get as much relations of the form $S(i,j,k,l)=1$ for \emph{small} $i,j,k,l$, as such relations put a lot of restrictions to the first 
entries of an $L$-algebra multiplication table. 

\subsubsection{Searching for $L$-algebras on a fixed $3$-approximate semilattice}

Given a function $S(x_1,x_2,x_3,y)$ for a separating $3$-approximating semilattice on $\{1,\ldots,n-1 \}$, we want to construct multiplication tables of $L$-algebras on $\{1,\ldots,n \}$ that are consistent with the approximate semilattice structure. This means that
\begin{enumerate}
    \item $M_{i,j} = n \Longleftrightarrow (i=j) \textnormal{ or } (j = n)$. This is just the discreteness condition.
    \item For $i,j,k \in \{1,\ldots,n-1 \}$, $S(i,j,j,k) = 1$ holds if and only if $(i \wedge j) \cdot k = e$ which translates to the equation $M_{M_{i,j},M_{i,k}}=n$. As we are looking at discrete $L$-algebras, this is equivalent to $(M_{i,j} =M_{i,k})$ or 
    $(i=k)$.
    \item for $i,j,k,l \in \{1,\ldots, n-1 \}$, $S(i,j,k,l) = 1$ holds if and only if 
    \[
    (i \wedge j \wedge k) \cdot l = e,
    \]
    which translates to the equation $M_{M_{M_{i,j},M_{i,k}},M_{M_{i,j},M_{i,l}}}=n$. Due to discreteness, this is equivalent to $(M_{M_{i,j},M_{i,k}} = M_{M_{i,j},M_{i,l}})$ or $(M_{i,j}=M_{i,l})$ or $(i=l)$.
    \item For $g \in \Sym_{n-1}$ such that $S(i,j,k,l) = S(g(i),g(j),g(k),g(l))$ for all $i,j,k,l \in \{1,\ldots,n-1 \}$, we have $M \leq_{\mathrm{lex}} M^g$
\end{enumerate}

\begin{proof}[Proof of Proposition \ref{pro:discrete}]
The conditions mentioned before can easily be translated to constraints. 
In total, when distinguishing by separating $3$-approximate semilattices 
of size 7 that fulfill restricted lower semimodularity, we get 73 cases that can be treated separately. 
This leads to the complete calculation of non-isomorphic discrete $L$-algebras of size eight. 

The construction of the database  
of isomorphism classes of 
discrete $L$-algebras of size eight 
took about four days, see Table \ref{tab:runtime}.
\end{proof}

%
%
%
%
%
%

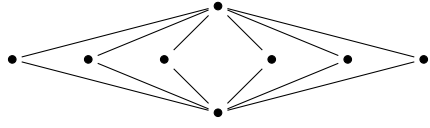
\begin{figure}
\begin{tikzpicture}
  \node [label=above:{}, label=below:{} ] (n8)  {} ;
  \node [above right of=n8,label=below:{}, label=below:{} ] (n4)  {} ;
  \node [above left of=n8,label=below:{}, label=below:{} ] (n5)  {} ;
    \node [left of=n5,label=below:{}, label=below:{} ] (n2)  {} ;
  \node [right of=n4,label=below:{}, label=below:{} ] (n3)  {} ;
  \node [left of=n2,label=below:{}, label=below:{} ] (n6)  {} ;
    \node [right of=n3, label=below:{}, label=below:{} ] (n1)  {} ;
  \node [above right of=n5, label=above:{}, label=below:{} ] (n7)  {} ;
  
   \draw [fill] (n1) circle [radius=.5mm];
   \draw [fill] (n2) circle [radius=.5mm];
   \draw [fill] (n3) circle [radius=.5mm];
   \draw [fill] (n4) circle [radius=.5mm];
   \draw [fill] (n5) circle [radius=.5mm];
   \draw [fill] (n6) circle [radius=.5mm];
   \draw [fill] (n7) circle [radius=.5mm];
   \draw [fill] (n8) circle [radius=.5mm];
   \draw (n2)--(n8);
   \draw (n1)--(n8);
   \draw (n3)--(n8);
   \draw (n4)--(n8);
   \draw (n5)--(n8);
   \draw (n6)--(n8);
   \draw (n2)--(n7);
   \draw (n1)--(n7);
   \draw (n3)--(n7);
   \draw (n4)--(n7);
   \draw (n5)--(n7);
   \draw (n6)--(n7);
\end{tikzpicture}
\caption{The diamond poset of eight elements.}
\label{fig:diamond}
\end{figure}

\subsection{Non-discrete $L$-algebras}

\begin{pro}
\label{pro:non_discrete}
    Up to isomorphism, there are 89,712,587 non-discrete $L$-algebras
    over a poset distinct from the diamond of Figure \ref{fig:diamond}.
\end{pro}

\begin{proof}
In the case that the natural order on $X^{\prime} = X \setminus \{ e \}$ is non-discrete then our approach can be seen as working with separating $1$-approximate semilattices, that is, posets:
\begin{enumerate}
    \item Construct all non-isomorphic posets on the set $\{1,\ldots,n-1\}$.
    \item For each isomorphism type of poset on $\{1,\ldots,n-1\}$, add another element $n$ and search for all non-isomorphic  $L$-algebras on
    the set $\{1,\ldots,n-1\}$ whose logical unit is $n$ and whose natural order restricts to the given partial order on $\{1,\ldots,n-1\}$.
\end{enumerate}

A few words on the construction of the posets: we use an easy python script to first construct all posets of size $n-1$ by starting with a poset of size $1$ and repeatedly adding new minimal elements by declaring an upwards-closed set from the recent poset as the set of elements lying above the new element. The size of the new constructed upset should be at least as big as the upsets that have already been constructed. We then filter out isomorphic posets by demanding that the constructed poset should be the least with respect to some lexicographic condition. We then encode the poset structure to a zero-one-matrix in a way that will be explained below.

Let $X$ be an $L$-algebra and $x,y,z \in X$. If $x\leq y$ with respect to the natural order in $X$, then
\[
x \cdot z = e \cdot (x \cdot z) = (x \cdot y) \cdot (x \cdot z) = (y \cdot x) \cdot (y \cdot z).
\]
Also recall, that for any elements $x,y,z \in X$, the implication $y \leq z \Rightarrow x \cdot y \leq x \cdot z$ holds.

Therefore, given a partial order relation on $\{1, \ldots,n-1 \}$, encoded as a matrix 
$P \in \{ 0,1 \}^{(n-1) \times (n-1)}$ with
\[
P_{i,j} = \begin{cases}
1 & i \leq j, \\
0 & \textrm{otherwise},
\end{cases}
\]
we use the following constraints, additionally to the $L$-algebra axioms:
\begin{enumerate}
    \item $M_{i,j} = n \Longleftrightarrow P_{i,j} = 1$ for $i,j \in \{1, \ldots,n-1 \}$,
    \item $P_{j,k} = 1 \implies M_{M_{i,j},M_{i,k}} = n$, for $i,j,k \leq n-1$
    \item $P_{i,j} = 1 \implies M_{i,k} = M_{M_{j,i},M_{j,k}}$ for $i,j,k \leq n-1$,
    \item if $g \in \Sym_{n-1}$ is such that $P_{i,j} = P_{g(i),g(j)}$ holds for all $i,j \leq n-1$ 
    , then $M \leq_{\textrm{lex}} M^g$ 
\end{enumerate}

These conditions can directly translated to constraints that enable us to enumerate \emph{almost} all non-isomorphic $L$-algebras whose natural order is non-discrete, except for the case of a \emph{diamond order}, whose restriction to the set $\{1, \ldots, n-1 \}$ is given by $i \leq_{\textrm{diam}} j \Longleftrightarrow i = j \textnormal{ or } i = n-1$.
\end{proof}

\begin{pro}
\label{pro:diamond}
    Up to isomorphism, there
    are 138,219,543 $L$-algebras of size eight 
    on the diamond poset of Figure \ref{fig:diamond}. 
\end{pro}

\begin{proof}
    This case turns out to be the hardest amongst all non-discrete posets. 
    We split up the case as follows. Let $i,j \in \{1, \ldots,n-2 \}$. If, in the $\wedge$-closure $C(X)$, the condition $i \wedge j \leq k$ holds for some $k \in \{1, \ldots, n-2 \}$, then this can be expressed as an equation
\[
(i \wedge j) \cdot k = e \Longleftrightarrow (i \cdot j) \cdot (i \cdot k) = e.
\]
     
     We use an approach which can be seen as a really coarse way of using $2$-approximate semilattices, for as to our knowledge, there is no theorem that describes the semilattice structure of $C(X)$ well enough to restrict the search to a small quantity of $2$-approximate semilattices. So we first construct matrices $R \in \{ 0,1\}^{(n-2) \times (n-2)}$ such that $R_{i,i} = 0$ and $R_{i,j} = R_{j,i}$ for all $i,j \in \{1, \ldots,n-2 \}$. These matrices parametrize all couples of $i,j \in \{ 1, \ldots,n-2 \}$ with the property that $i \wedge j \leq k$ holds for all $k \in \{1, \ldots,n-2 \}$ in the $\wedge$-closure of the final $L$-algebra. If we look at two such matrices as being equivalent if they only differ by a simultaneous relabelling of rows and columns, then we get 156 cases that can be treated as follows:

    Given a symmetric zero-one-matrix $R$ with only zeros on the diagonal, we have the following conditions
    \begin{enumerate}
      \item [(4')] the lexicographical symmetry breaking is replaced as follows: if $g \in \Sym_{n-2}$ is such that $R_{i,j} = R_{g(i),g(j)}$ holds for all $i,j \leq n-1$ 
    , then $M \leq_{\textrm{lex}} M^g$ 
    \item [(5)] for all $i,j \in \{1,\ldots,n-2 \}$, we have $R_{i,j}=1$ if and only if $M_{M_{i,j},M_{i,k}}=n$ holds for all $k \in \{1, \ldots ,n-2 \}$.
\end{enumerate}

The required 
156 cases were treated in about four hours, see Table \ref{tab:runtime} for more
information. 
\end{proof}

\section{Linear $L$-algebras}
\label{linear}

Recall that under 
the partial order \eqref{eq:poset} 
left-multiplications are monotone, i.e. if $y\leq z$, then $x \cdot y \leq x \cdot z$ for all $x\in X$.

\begin{defn}
\label{defn:linear}
An $L$-algebra $X$ is said to be \emph{linear} if 
$x \leq y$ or $y \leq x$ for all $x,y \in X$.
\end{defn}

In the case of linear $L$-algebras   
\eqref{eq:poset} is indeed a total order. 




For $n\geq2$ let $\lambda(n)$ be the number of non-isomorphic linear $L$-algebras with $n$ elements. A look into
the sequence A000110 of \cite{oeis} suggests that $\lambda(n)=B(n-1)$, where $B(m)$ is the $m$-th \emph{Bell number} that counts the number of ways to partition a set with $m$ elements into non-empty subsets.


\begin{defn}
We call an element $x$ in an $L$-algebra $X$ \emph{invariant} if $y \cdot x = x$ holds for all $y > x$. 
\end{defn}

 
 
\begin{lem} 
\label{lem:linear}
Let $n \geq 1$ and let $X$ be an $L$-algebra on the ordered set 
\[
L_n = \{ x_0 > x_1 > \ldots > x_{n-1} \}.
\]
Then the following statements hold. 
\begin{enumerate}
     \item If $x \geq y > z$, then $x \cdot y > x \cdot z$ for all $x,y,z \in X$.
     \item $x \cdot y \geq y$ holds for all $x,y \in X$.
     \item If $x \leq y$, then $x \cdot z \geq y \cdot z$ for all $x,y,z \in X$.
     \item If $p\ne e$ is the smallest invariant element of $X$, then $p \cdot x > x$ for all $x \leq p$.
 \end{enumerate}
 \end{lem}
 
 \begin{proof} (1) Let $x \geq y > z$, then
     \[
(x \cdot y) \cdot (x \cdot z) = (y \cdot x) \cdot (y \cdot z) = e \cdot (y \cdot z) = y \cdot z \neq e.
\]
This implies that $x \cdot y \not\leq x \cdot z$. As the order on $X$ is total, this implies $x \cdot y > x \cdot z$.

(2) Without loss of generality, we may assume that $x>y$. By (1), for each 
$x\in X$ the map $\varphi_x\colon [x_{n-1},y] \to [x_{n-1},x \cdot y]$, $z \mapsto x \cdot z$,
is a well-defined monotone embedding. Therefore 
\[
|[x_{n-1},y]|=|\varphi_x([x_{n-1},y])|\leq |[x_{n-1},x \cdot y]|.
\]
Since these intervals are finite, it follows that 
$x \cdot y \geq y$.

(3) Let $x,y,z \in X$ be such that $x \leq y$. By (2), 
\[
x \cdot z = e\cdot (x\cdot z)=
(x \cdot y) \cdot (x \cdot z) = (y \cdot x) \cdot (y \cdot z) \geq y \cdot z.
\]

(4) Assume that there exists a maximal $x$ with $x < p$ and $p \cdot x = x$. As $x$ is not invariant, there is a greatest element $y$ such that 
$y>x$ and $y\cdot x > x$. Then $y<p$, as $y\geq p$ and (3) imply
\[
x=p\cdot x\geq y\cdot x>x, 
\]
a contradiction. The maximality of $x$ implies that $p\cdot y>y$. 
Since $y$ is maximal and $p\cdot y>y$, $(p\cdot y)\cdot x=x$. By using
the cycloid equation,
\[
x = (p \cdot y) \cdot (p \cdot x) = (y \cdot p) \cdot (y \cdot x) = y \cdot x > x,
\]
a contradiction. 
Therefore $p \cdot x > x$.
\end{proof}

We can now prove that each linear $L$-algebra can be built up inductively by adding a new smallest element $x^{\prime}$ and declaring an admissible value for $p \cdot x^{\prime}$, where $p$ is the smallest invariant element in the old $L$-algebra.

\begin{pro}
\label{pro:linear}
Let $p$ be the smallest invariant element in $X$. Let $c$ be an element in the superset $L_{n+1} \supseteq L_n$ such that $p \cdot x_{n-1} > c$. Then there is a unique $L$-algebra $X^{\prime}$ on $L_{n+1}$ such that the canonical embedding $L_n \hookrightarrow L_{n+1}$ embeds $X$ as a sub $L$-algebra of $X^{\prime}$ and such that $p \cdot x_n = c$.
\end{pro}

As $p \cdot x_n$ must necessarily be smaller than $p \cdot x_{n-1}$ (part (1) of Lemma \ref{lem:linear}), 
the proposition describes the construction of \emph{each} 
extension to an $L$-algebra on $L_{n+1}$.

\begin{proof}
We look at two cases, namely $c = x_n$ or $c > x_n$. Note that we must define $a \cdot b$ only for $a \in L_n$ and $b = x_n$, as $x_n \cdot a = e$ is forced by the order of $P_{n+1}$.

Assume first that $c = x_n$. 
By Lemma \ref{lem:linear}, setting $a \cdot x_n := x_n$ for $a\ne x_n$ defines the only possible extension to an $L$-algebra structure on $L_{n+1}$ such that $p \cdot x_n = x_n$. We now check that this definition is indeed compatible with the cycloid equation, meaning that
\[
(x\cdot y)\cdot (x\cdot z)=
(y\cdot x)\cdot (y\cdot z)
\]
holds for all $x,y,z \in L_{n+1}$. Without loss of generality 
may assume that $x,y,z$ are different. 
If $z=x_n$, the result is trivially true. Now we may assume
that $x>y$. If $y=x_n$, then 
\[
(x \cdot x_n) \cdot (x \cdot z) = x_n \cdot (x \cdot z) = e
=e\cdot e=
(x_n \cdot x) \cdot (x_n \cdot z). 
\]

This proves that this indeed defines an extension to an $L$-algebra on $L_{n+1}$. It remains to show that this is the only possibility. Therefore, let $X^{\prime}$ be an extension to an $L$-algebra on $L_{n+1}$ with $p \cdot x_n = x_n$, where $p$ is the smallest invariant element of $X$.
    
By Lemma \ref{lem:linear}(1), for $a > p$, the 
map $[x_n,p] \to [x_n,a \cdot p] = [x_n,p]$, $x\mapsto a\cdot x$, is a monotone embedding, i.e. an isomorphism of ordered sets. It follows that $a \cdot x_n = x_n$.
We now know that $b \cdot x_n = x_n$ for all $b \geq p$. Assume now 
that $p \geq a > x_n$ and that $b \cdot x_n=x_n$ holds for all $b > a$. Since 
$a\cdot p=e$ and $a\cdot x_n=e$, 
\begin{gather*}
   a\cdot x_n=e\cdot (a\cdot x_n)=(a\cdot p)\cdot (a\cdot x_n)
    \shortintertext{and hence}
    a \cdot x_n = (a \cdot p) \cdot (a \cdot x_n) = (p \cdot a) \cdot (p \cdot x_n) = (p \cdot a)\cdot x_n = x_n,
\end{gather*}
as $p\cdot a>a$ by Lemma \ref{lem:linear}(4). 

We now assume that $c > x_n$. 
Again we need to define $a \cdot x_n$ when $a > x_n$. We claim that the only possible definition is given by
\[
a \cdot x_n = \begin{cases}
(p \cdot a) \cdot c & a \leq p, \\
x_n & a > p.
\end{cases}
\]
The necessity is easily shown: if $p$, $x$ and $y$ are elements in a linear $L$-algebra such that $p \geq x$ and $p\geq y$, the cycloid equation implies that
\[
x \cdot y = e \cdot (x \cdot y) = (x \cdot p) \cdot (x \cdot y) = (p \cdot x) \cdot (p \cdot y).
\]
So the first equation follows from demanding $p \cdot x_n = c$. On the other hand, as $p$ is invariant, for all $x > p$ the map given by left-multiplication by $x$ is an automorphism of the ordered set $[x_n,p]$. In particular, $x \cdot y = y$ for all $y \in [x_n,p]$.

Now, we check that these assignments do not violate the cyclic condition. Let $z = x_n$ and $x,y > p$, then the left side of the equation is
\[
(x \cdot y) \cdot (x \cdot x_n) = (x \cdot y)\cdot x_n = x_n,
\]
as $x \cdot y\geq y > p$ by Lemma \ref{lem:linear}(2). Similarly, $(y \cdot x) \cdot (y \cdot x_n) = x_n$.

Now let $z = x_n$ and $x > p \geq y$. Then, as before, 
\begin{align*}
&(x \cdot y) \cdot (x \cdot x_n) = y \cdot x_n,\\
&(y \cdot x) \cdot (y \cdot x_n) = e \cdot (y \cdot x_n) = y \cdot x_n.
\end{align*}
For $z = x_n$ and $x,y \leq p$, we use $x \cdot y = (p \cdot x) \cdot (p \cdot y)$ to calculate
\begin{align*}
(x \cdot y) \cdot (x \cdot x_n) & = \left( (p \cdot x) \cdot (p \cdot y) \right) \left( (p \cdot x) \cdot c \right) \\
& = \left( (p \cdot y) \cdot (p \cdot x) \right) \left( (p \cdot y) \cdot c \right) \\
& = (y \cdot x) \cdot (y \cdot x_n).
\end{align*}
Finally, let $x = x_n$ (the case where $y = x_n$ follows by symmetry). 
Then 
\[
(x_n \cdot y) \cdot (x_n \cdot z) = e \cdot e = e.
\]
If $y > p$, then $y \cdot x_n = x_n$ and hence 
$(y \cdot x_n) \cdot (y \cdot z)  = e$. On the other hand, if $y \leq p$, then
\[
y \cdot z \geq y \cdot x_{n-1} = (p \cdot y) \cdot (p \cdot x_{n-1}) \geq (p \cdot y) \cdot c = y \cdot x_n
\]
which finally implies that $(y \cdot x_n) \cdot (y \cdot z)  = e$.
\end{proof}

\begin{thm}
\label{thm:linear}
For all positive integer $n$, $\lambda(n)=B(n-1)$.
\end{thm}

\begin{proof}
Note that $\lambda(n)$ is the number of $L$-algebras on a \emph{fixed} $n$-element set $L_n$. As each isomorphism of $L$-algebras is also an isomorphism of posets and as finite totally ordered sets do not have any nontrivial automorphisms, any two different $L$-algebra structures on $L_n$ are non-isomorphic.

We proceed by induction on $n$. 
Clearly,  $\lambda(1) = B(0) = 1$. Let $n > 0$ and assume that $\lambda(k) = B(k-1)$ holds for all $k < n$. For $m\in\{1,\dots,n-1\}$ let $\lambda(n,m)$ be the number of all linear $L$-algebras on the ordered set $L_n$ whose smallest invariant element is $x_m$. We will now show that 
\[
\lambda(n,m) = \binom{n-1}{m}B(m).
\]
To this end, let $X$ be an $L$-algebra on $L_n$ whose smallest invariant element is $x_m$. 
Let 
\begin{gather*}
X_m = \{x_i \in X : i < m \} \subseteq X
\shortintertext{and the subset}
X^m = \{x_m \cdot x_j : j > m \} \subseteq L_{n-1}.
\end{gather*}
We claim that the assignment $X \mapsto (X_m,X^m) $ is a bijective correspondence 
between $L$-algebras on $L_n$ with smallest invariant element $x_m$ and pairs $(X',A)$, where 
$X^{\prime}$
is an $L$-algebra 
on $L_m$ and $A$ is an $(n-m-1)$-subset of $L_{n-1} \setminus \{ e \}$.

Let $X^{\prime}$ be an $L$-algebra on $L_m$ and $A \subseteq L_{n-1}$ be 
an $(n-m-1)$-subset of $L_{n-1}$. By Proposition \ref{pro:linear}, 
there is a unique $L$-algebra $X^{(0)}$ on $L_{m+1}$ that has $x_m$ as an invariant element and restricts to $X^{\prime}$ on $L_m$. Assume that 
\[
A = \{ a_1 > a_2 > \ldots > x_{n-m-1} \}
\]
and note that from $a_{n-m-1} \geq x_{n-2}$ it follows that $a_i \geq x_{m-1+i} > x_{m+i}$ for all $i$.
Clearly, $x_m$ is the smallest invariant element of $X^{(0)}$ and $x_m \cdot x_m = e > a_1$. By Proposition \ref{pro:linear}, there is a unique $L$-algebra $X^{(1)}$ on $L_{m+2}$ such that $p \cdot x_{m+1} = a_1$ and which restricts to $X^{(0)}$ on $L_{m+1}$. Note that $x_m$ is still the smallest invariant element in $X^{(1)}$ as $a_1 > x_{m+1}$. By repeated use of the proposition, one can show inductively that there is a unique sequence of $L$-algebras $X^{(i)}$ on $L_{m+1+i}$ for all $i \leq n-m-1$ such that $X^{(i)}$ restricts to $X^{(i-1)}$ on $L_{m+i}$, whose smallest invariant element is $x_m$ and such that $x_m \cdot x_{m+i} = a_i> x_{m+i}$.

We conclude that there is a unique $L$-algebra on $L_n$ that has $x_m$ as its smallest invariant element, restricts to $ X^{\prime} $ on $L_m$ and fulfils 
\[
\{ x_m \cdot x_j: j > m \} = A.
\]
By induction, it follows that
\[
\lambda(n,m) = \binom{n-2}{n-m-1} B(m-1) = \binom{n-2}{m-1}B(m-1).
\]
Using a well-known recursion formula for Bell numbers (see for example
\cite{MR161805}), we conclude that 
\[
\lambda(n) = \sum_{m=1}^{n-1} \lambda(n,m) = \sum_{m=1}^{n-1} \binom{n-2}{m-1}B(m-1) = B(n-1).\qedhere
\]
\end{proof}


Unfortunately, despite using the recursive construction of finite linear $L$-algebras, this recursion at first glance does not seem to translate into the recursive construction of partitions that leads to the used recursion formula of Bell numbers. We, therefore, pose the following problem: 

\begin{problem}
Find an explicit bijection between $L$-algebras on $L_n$ and partitions of the set $\{0,\ldots,n-1 \}$.
\end{problem}

\section{Regular $L$-algebras}
\label{regular}


\begin{defn}
An $L$-algebra $X$ is called \emph{semiregular} if
\[
\left( (x \cdot y) \cdot z \right) \cdot \left( (y \cdot x) \cdot z \right) = \left( (x \cdot y) \cdot z \right) \cdot z
\]
for all $x,y,z\in X$.
\end{defn}



\begin{pro}[Rump] 
\label{pro:(x.y).x=e.in_semiregular_algebra}
If $X$ is a semiregular $L$-algebra, then $x \leq y \cdot x$
for all $x,y\in X$. 
\end{pro}

\begin{proof}
This is \cite[Proposition 14]{MR2437503}.
\end{proof}

The proposition implies that in a semiregular $L$-algebra $X$, all upsets 
\[
x^{\uparrow} = \{ y \in X : y \geq x \}
\]
are sub-$L$-algebras.


\begin{defn}
\label{defn:regular}
A \emph{regular} $L$-algebra is a semiregular $L$-algebra $X$ 
such that for all $x,y \in X$ with $x \leq y$ there exists a $z \geq x$ such that $z \cdot x = y$.
\end{defn}

A big family of regular $L$-algebras can be constructed from $\ell$-groups. In the following, we will collect the necessary definitions and facts.

First of all, an \emph{ordered group} is a pair $(G,\leq)$ of a group $G$ and a partial 
order relation $\leq$ on $G$ that is invariant under 
multiplication in the sense that for all $x,y,z \in G$, the implication
\[
x \leq y \Rightarrow (xz \leq yz) \  \mathrm{ and }\  (zx \leq zy)
\]
holds. As usual, we will just say that $G$ itself is an ordered group whenever it is clear from the context that $G$ is equipped with an appropriate order relation.
If $G$ is an ordered group, we call $G$ a \emph{lattice-ordered group} (\emph{$\ell$-group}, for short) when $G$ is a lattice under $\leq$. This means that for any elements $x,y \in G$, there is a \emph{meet} $x \wedge y$, i.e. a greatest element $z$ such that $z \leq x$ and $z \leq y$; also, there is a \emph{join} $x \vee y$, i.e. a least element $z$ such that $z \geq x$ and $z \geq y$. It can be shown 
that in an $\ell$-group $G$, 
for any $x,y,z \in G$, the following equations hold:
\begin{align*}
    (x \wedge y)z = xz \wedge yz &,\quad z(x \wedge y) = zx \wedge zy \\
    (x \vee y)z = xz \vee yz &, \quad z(x \vee y) = zx \vee zy.
\end{align*}

An important example of a right $\ell$-group is the additive group $\mathbb{Z}$, with $\leq$ being the natural order on the integers. The lattice operations are given by $x \wedge y = \min \{ x,y \}$ and $x \vee y = \max \{ x,y \}$, respectively, where $x,y \in \mathbb{Z}$.

$\ell$-groups give rise to a special class of $L$-algebras. 
Let $G$ be an $\ell$-group and define its \emph{negative cone} as $G^- = \{ g \in G : g \leq e \}$. Then $G^-$ has a canonical $L$-algebra structure given by the binary operation
\[
x \cdot y = (x \wedge y)x^{-1} = e \wedge yx^{-1}.
\]
In particular, $(G^-,\cdot)$ is a regular $L$-algebra, see \cite[page 2343]{MR2437503}.
Except for trivial examples, these $L$-algebras are infinite. However, as $x \cdot y \geq y$ holds in any regular $L$-algebra by Proposition \ref{pro:(x.y).x=e.in_semiregular_algebra}, it follows that each non-empty, upwards-closed subset $X \subseteq G^-$ is closed under the $L$-algebra operation and is therefore a regular $L$-algebra by restriction. Here, a subset $X \subseteq G^-$ is called \emph{upwards-closed} if for all $x,y \in G^-$ with $x \in X$ and $y \geq x$, we also have that $y \in X$. Although $G^-$ is almost always infinite, a sub-$L$-algebra thereof can be finite in a few more cases.

Before stating Rump's embedding theorem for regular $L$-algebras, we need the notion of dense element (which can be defined in arbitrary $L$-algebras). 

\begin{defn}
An element $x$ in an $L$-algebra $x$ is called \emph{dense} whenever there is an $y \in X$ with $x \geq y$ and $x \cdot y = y$.
\end{defn}

\begin{thm} \label{thm:rump_embedding_theorem}
An $L$-algebra $X$ can be embedded as a sub-$L$-algebra in the negative cone of some $\ell$-group $G$ if and only if $X$ is regular and has no dense elements aside from $e$. For such an $L$-algebra, the $\ell$-group $G$ can be chosen in such a way that $G^-$ is generated as a monoid by $X$, in which case, $G$ is unique up to isomorphism of $\ell$-groups.
\end{thm}

\begin{proof}
See \cite[Theorem 4.]{MR2437503}. The second part essentially follows from Theorem 3. which implies that the monoid and lattice structures of $G^-$ are determined by $X$ - which implies that those of $G$ are, as well.
\end{proof}

If $(G_i)_{i \in I}$ is a (not necessarily finite) family of $\ell$-groups, we can form their \emph{cardinal sum}
\[
\bigoplus_{i \in I} G_i := \left\{ (g_i)_{i \in I} \in \prod_{i \in I} G_i \ : \ g_i = 0 \textnormal{  for all but finitely many } i \in I \right\}
\]
which again is an $\ell$-group under coordinate-wise group and lattice operations. 
Two important examples of such $\ell$-groups are $\mathbb{Z}^m = \oplus_{i = 1}^m \mathbb{Z}$ and 
$\mathbb{Z}^{\omega} = \oplus_{i=1}^{\infty} \mathbb{Z}$.

\begin{defn}
An \emph{infinite-dimensional Young-diagram} is a finite, non-empty and upwards-closed subset in $ (\mathbb{Z}^{\omega})^-$.
\end{defn}

\begin{defn}
Two infinite-dimensional Young-diagrams are said to be \emph{conjugate} 
if they can be converted into one another by a permutation of the coordinates of~$\mathbb{Z}^{\omega}$.
\end{defn}

As an upwards-closed set in the negative cone of an $\ell$-group, each infinite-dimensional Young-diagram becomes an instance of a regular $L$-algebra after restricting the $L$-algebra operations from $(\mathbb{Z}^{\omega})^-$ to $X$. Also, conjugate infinite-dimensional Young-diagrams have isomorphic $L$-algebra structures, for permuting the coordinates of $\mathbb{Z}^{\omega}$ is clearly an automorphism of $\ell$-groups and therefore also induces an automorphism of the respective $L$-algebra structure. This implies that conjugate infinite-dimensional Young-diagrams are related by an isomorphism of $L$-algebras.

When enumerating regular $L$-algebras 
we observed that the first numbers are equal to those in sequence A119268 of \cite{oeis}  which counts infinite-dimensional Young-diagrams up to conjugacy. This suggested that infinite-dimensional Young-diagrams might actually be the same as finite regular $L$-algebras. 

\begin{lem}
\label{lem:dense}
    Let $X$ be a finite regular $L$-algebra and $x\in X$.
    If $x$ is dense, then $x=e$. 
\end{lem}

\begin{proof}
    Let $x \geq y$ be elements in $X$ such that $x \cdot y = y$. By Proposition \ref{pro:(x.y).x=e.in_semiregular_algebra}, the map $y^{\uparrow} \to y^{\uparrow}$, $x \mapsto x \cdot y$, is well-defined and, due to regularity, surjective. Since $X$ is finite, this map is also injective which implies that the element $y \in y^{\uparrow}$ has exactly one preimage, which must be $e$. 
\end{proof}

\begin{thm}
\label{thm:regular}
    There exists a bijective correspondence between finite regular $L$-algebras
    and infinite-dimensional Young diagrams. 
\end{thm}

\begin{proof}
By Theorem \ref{thm:rump_embedding_theorem}, there is an $\ell$-group $G$ with an embedding $X \hookrightarrow G^-$ that identifies $X$ as a sub-$L$-algebra of $G^-$ and a generating set thereof. For sake of simplicity, we will identify $X$ with its image in $G^-$. We want to show that $G \simeq \mathbb{Z}^m$ and that $X$ is upwards-closed in $G^-$ under this embedding.

Recall, that for $x,y \in X$, the relation $x\prec y$ means that $x< y$ and there
is no $z\in X$ such that $x<z<y$.  
Let $\mathrm{Max}(X) := \{ x \in X : x \prec e \}$ be the set of elements lying directly under $e$. 

We claim that $G^-$ is generated by $\mathrm{Max}(X)$. Let $M \subseteq G^-$ be the submonoid generated by $\mathrm{Max}(X)$. If there was an element in $X \setminus M$, there would also be a biggest one, say $y \in X \setminus M$. Necessarily, there exists $x \in \mathrm{Max}(X)$ with $x \geq y$. By Proposition \ref{pro:(x.y).x=e.in_semiregular_algebra} and Lemma \ref{lem:dense}, $x \cdot y > y$. Since $x\in\mathrm{Max}(X)$, $x \cdot y \in M$. Also, we calculate in $G^-$ that $y = x \wedge y = (x \cdot y)x \in M$, a contradiction. 
We conclude that $X \subseteq M$. As $G^-$ is generated by $X$, it follows that $G^-$ is generated by 
$\mathrm{Max}(X)$.

Let now $x,y \in \mathrm{Max}(X)$. By Proposition \ref{pro:(x.y).x=e.in_semiregular_algebra}, $x \cdot y \geq y$. Here, $x \cdot y = e$ is only possible when $x \leq y$, i.e. when $x = y$. Therefore, $x \cdot y = y$ for $x,y \in \mathrm{X}$ with $x \neq y$. In this case, one can calculate in $G^-$ that $x \wedge y = (x \cdot y)x = yx$ which is the same as $y \wedge x = xy$. These relations show that we have a decomposition of the $\ell$-group $G$ as a cardinal sum
\[
G \simeq \bigoplus_{x \in \mathrm{Max}(X)} \mathbb{Z}
\]
where $\mathrm{Max}(X)$ is identified with the elements in the right group that are $-1$ in exactly one coordinate and $0$ in all others. 

It remains to show that $X$ is upwards-closed in $G^-$. It clearly suffices to show that for $y \in X$ and all $x \in G^-$ with $x \succ y$, we also have $x \in X$. In such a case, $z = yx^{-1} \in \mathrm{Max}(X)$. As we have already seen that $G$ is abelian, we also have $z = x^{-1}y \geq y$. We can now calculate in $G^-$
\[
z \cdot y = (z \wedge y)z^{-1} = y(y^{-1}x) = x.
\]
As $X$ is a sub-$L$-algebra of $G^-$, this implies that $x \in X$.

We have thus shown that each finite regular $L$-algebra $X$ is isomorphic to an upwards-closed subset in the negative cone of some $\ell$-group $\mathbb{Z}^m$. Under a coordinate-wise embedding $\mathbb{Z}^m \hookrightarrow \mathbb{Z}^{\omega}$, $X$ is identified with an infinite-dimensional Young-diagram. By the uniqueness part of Theorem \ref{thm:rump_embedding_theorem}, the embedding groups of isomorphic regular $L$-algebras without nontrivial dense elements are also isomorphic, in the sense of $\ell$-groups. Is is easily seen that each isomorphism of two $\ell$-groups of the form $\mathbb{Z}^m$ is given by a permutation of coordinates. This implies that each isomorphism class of finite regular $L$-algebras can be identified with a unique conjugacy class of infinite-dimensional Young-diagrams and that only conjugate Young-diagrams are isomorphic as $L$-algebras.
\end{proof}

\section*{Acknowledgments}

We thank Wolfgang Rump for comments and suggestions. 
This research was supported through the program ``Oberwolfach Research Fellows" by the Mathematisches Forschungsinstitut Oberwolfach in 2022. Mench\'on was partially supported by the National Science Center (Poland), grant number~2020/39/B/HS1/00216; for the purpose of Open Access, the author has applied a CC-BY public copyright licence to any Author Accepted Manuscript (AAM) version arising from this submission.
Vendramin was supported in part by OZR3762 
of Vrije Universiteit Brussel. 

\bibliographystyle{abbrv}
\bibliography{refs}

\end{document}